\documentclass[11pt,a4paper]{article}
\usepackage[utf8]{inputenc}
\usepackage{epsfig,graphicx,graphics}
\usepackage{amsmath}
\usepackage{amssymb}

\usepackage{amsthm}
\usepackage{amsmath}
\usepackage{amsfonts}
\usepackage{geometry}
\usepackage{epsfig}

\usepackage{color}
\usepackage{amssymb}
\usepackage{graphicx}
\usepackage{amsmath,amsfonts,amsthm,amstext}
\usepackage{mathrsfs}
\usepackage{hyperref}

\usepackage{tikz}

\newtheorem{teo}{Theorem}[section]
\newtheorem{lema}{Lemma}[section]
\newtheorem{cor}{Corollary}[section]
\newtheorem{propo}{Proposition}[section]
\theoremstyle{definition}
\newtheorem{defi}{Definition}[section]

\newtheorem{rek}{Remark}[section]

\newcommand{\supp}{\operatorname{supp}}

\newcommand{\essinf}{\operatorname{ess\,inf}}

\newcommand{\esssup}{\operatorname{ess\,sup}}

\newcommand{\diam}{\operatorname{diam}}

\newcommand{\ve}{\varepsilon}
\newcommand{\N}{\mathbb{N}}
\newcommand{\R}{\mathbb{R}}
\newcommand{\Z}{\mathbb{Z}}
\newcommand{\M}{\mathcal{M}}

\newenvironment{proof1}[1][\textit{\textbf{Proof \em(Theorem~\ref{central1})}}]{\textit{#1.} }{\hfill $\Box$}

\newenvironment{proof3}[1][\textit{\textbf{Proof \em(Theorem~\ref{GExpansive})}}]{\textit{#1.} }{\hfill $\Box$}

\newenvironment{proof4}[1][\textit{\textbf{Proof \em(Theorem~\ref{entrinf1})}}]{\textit{#1.} }{\hfill $\Box$}

\newenvironment{proof5}[1][\textit{\textbf{Proof \em(Theorem~\ref{entrinf2})}}]{\textit{#1.} }{\hfill $\Box$}

\begin{document}

\bigskip

\title {On the generic behavior of the metric entropy, and related quantities, of uniformly continuous maps over Polish metric spaces}
\date{}
\author{Silas L. Carvalho \thanks{Work partially supported by FAPEMIG (a Brazilian government agency; Universal Project 001/17/CEX-APQ-00352-17)} ~~and~ Alexander Condori \thanks{ Work  partially  supported  by  CIENCIACTIVA C.G. 176-2015}
}
\maketitle

{ \small \noindent $^{*}\,$Instituto de Ciências Exatas (ICEX-UFMG). Av. Pres. Antônio Carlos 6627, Belo Horizonte-MG, 31270-901, Brasil. \\ {\em e-mail:
silas@mat.ufmg.br 

\em
{\small \noindent $^{\dag\,}$Departamento de Matemática y Física (Universidad Nacional de San Cristóbal de Huamanga-UNSCH). Portal Independencia 57–Huamanga–Ayacucho.
Ayacucho 05003,
Per\'u. \\ {\em e-mail:
alexander.condori@unsch.edu.pe 
\em
\normalsize
\maketitle
\begin{abstract}
\noindent{
  In this work, we show that if $f$ is a uniformly continuous map defined over a Polish metric space, then the set of $f$-invariant measures with zero metric entropy is a $G_\delta$ set (in the weak topology). In particular, this set is generic if the set of $f$-periodic measures is dense in the set of $f$-invariant measures. This settles a conjecture posed by Sigmund in~\cite{Sigmund1974}, which states that the metric entropy of an invariant measure of a topological dynamical system that satisfies the periodic specification property is typically zero.

  We also show that if $X$ is compact and if $f$ is an expansive or a Lipschitz map with a dense set of periodic measures, typically the lower correlation entropy for $q\in(0,1)$ is equal to zero. Moreover, we show that if $X$ is a compact metric space and if $f$ is an expanding map with a dense set of periodic measures, then the set of invariant measures with packing dimension, upper rate of recurrence and  upper quantitative waiting time indicator equal to zero is residual.

Finally, we present an alternative proof of the fact that the set of expansive measures is a $G_{\delta\sigma}$ set in the set of probability measures $\M(X)$, if $X$ is a Polish metric space and if $f$ is uniformly continuous (this result was originally proved by Lee, Morales and Shin in \cite{Morales2018} for compact metric spaces). 
}
\noindent
\end{abstract}
{Key words and phrases}.  {\em Metric entropy, invariant measures,  expansive measures.}

\section{Introduction}

\subsection{Metric entropy of invariant measures}
Given a dynamical system $(X,f)$, where $X$ is a measurable space and $f:X\to X$ is a measurable function, the metric entropy (or just entropy) of an $f$-invariant measure (supposing that it exists) is related to the asymptotic growth rate of the loss of information with respect to the time evolution (that is, as one computes successive iteractions of $f$).

Namely, if $\mu$ is an $f$-invariant measure, then
\[h_{\mu}(f):=\sup h_{\mu} (f, \mathcal{Q}) =\sup \lim_{n \to \infty} \frac{1}{n} H(\mathcal{Q}\vee f^{-1}\mathcal{Q}\vee\cdots\vee f^{-n+1}\mathcal{Q}),
  \]
  where the supremum is taken over all the measurable partitions $\mathcal{Q}$ of $X$ such that $H(Q)<\infty$.

  So, in case $h_{\mu}(f)>0$, the degree of disorder (loss of information) that the transformation $f$ causes to some partition of $X$ grows exponentially with time. There is a large number of dynamical systems for which the set of invariant measures with positive metric entropy is dense (see, for example, \cite{Abdenur,Sigmund1971, Sigmund1972}). 

However, for numerous examples of dynamical systems, not only the set of invariant measures with zero metric entropy is dense, but it is also a $G_\delta$ set (being, therefore, generic;  
see \cite{Abdenur,Oprocha,Sigmund1970,Sigmund1971,Sigmund1974}, where this result is obtained through different methods and assumptions). Observe that if $h_\mu(f)=0$, then for each $\mathcal{Q}$ such that $H(Q)<\infty$, there exists an $N\in\N$ such that for each $n\ge N$, the dynamical partition $\mathcal{Q}_n:=\mathcal{Q}\vee f^{-1}\mathcal{Q}\vee\cdots\vee f^{-n+1}\mathcal{Q}$ is equal (up to an exponential grow) to $\mathcal{Q}_N$; that is, the degree of disorder caused by $f$ over the space $X$ is small (or equivalently, the initial partition $\mathcal{Q}$ does not undergo to many sub-atomizations by the successive application of $f$).

Thus, it is somewhat surprising that even for some dynamical systems which are very sensitive to the initial conditions (``chaotic'', in some sense; this is the case of expanding topologically mixing maps on intervals~\cite{Blokh1995,Gelfert,Oprocha,Sigmund1971, Sigmund1974}), there is a dense set (in the weak topology; see the discussion ahead) of invariant measures with zero metric entropy.  

Motivated by such results, we investigate in this work the typical behavior (in Baire's sense) of the entropy function when $X$ is a Polish metric space and $f:X\to X$ is a (uniformly) continuous function. 

Let $\mathcal{M}(X)$ be the space of all probability measures defined on $X$, endowed with the weak topology (that is, the coarsest topology for which the net $\{\mu_\alpha\}$ converges to $\mu$ if, and only if, for each bounded and continuous function $\varphi$, $\int \varphi d\mu_\alpha\rightarrow \int \varphi d\mu$). Since $X$ is Polish, $\mathcal{M}(X)$ is also a Polish metrizable space (see~\cite{Sigmundlibro}). We also denote by $\M(f)$ the set of $f$-invariant probability measures, by $\M_e(f)$ the set of $f$-ergodic measures and by $\M_p(f)$ the set of $f$-periodic measures, that is, the $f$-invariant measures which are supported on $f$-periodic (or $f$-closed) orbits; they are all endowed with the (metrizable) topology induced from $\M(X)$.

Our first result establishes that if $X$ is a Polish metric space and if $f:X\rightarrow X$ is a uniformly continuous function, then set of ergodic measures with zero metric entropy is a $G_\delta$ subset of $\M_e(f)$.

\begin{teo}
\label{central1}
Let $X$ be a  Polish space and let $f:X\to X$ be a uniformly continuous function. Then, the set
\[\mathcal{E}_0^e(f)=\{\mu\in \M_e(f)\mid h_\mu(f)=0\}\]
is a $G_\delta$ subset of $\M_e(f)$.
\end{teo}

Under the hypothesis that $f:X\rightarrow X$ is a continuous function, it follows from (a small adaptation of) Theorem~2.1 in~\cite{Parthasarathy1961} that $\M_e(f)$ is a $G_\delta$ subset of $\M(f)$. Thus, combining this remark with Theorem~\ref{central1},  one has the following result.

\begin{cor}
\label{corcentral1}
Let $X$ be a  Polish space and let $f:X\to X$ be a uniformly continuous function. Then, the set
\[\mathcal{E}_0(f)=\{\mu\in \M(f)\mid h_\mu(f)=0\}\]
is a $G_\delta$ subset of $\M(f)$.
\end{cor}

Since for any topological dynamical system $(X,f)$ (that is, $X$ is a compact metric space and $f:X\rightarrow X$ is a continuous function; in some situations, $f$ is required to be surjective), $\M_e(f)$ is a $G_\delta$ subset of $\M(f)$ (see~Proposition~5.7 in~\cite{Sigmundlibro}), one has from Remark~\ref{RGdeltaentropy} the following version of Corollary~\ref{corcentral1}. 

\begin{cor}
\label{corcentral2}
Let $(X,f)$ be a topological dynamical system. Then, $\mathcal{E}_0(f)$
is a $G_\delta$ subset of $\M(f)$.
\end{cor}

The idea used in the proof of Theorem~\ref{central1} combines a weaker version of Brin-Katok's Theorem (Corollary~\ref{BrinKatokC}) with an important characterization of the lower local entropy of an $f$-ergodic measure (presented in the proof of Theorem~\ref{Gdeltaentropy}; see also Lemma~\ref{Gdelta3}).

\begin{rek}\label{semicontinuity} We note that the result stated in Corollary~\ref{corcentral1} is trivial if the entropy function $\M(f)\ni \mu\mapsto h_\mu(f)\in[0,\infty]$ is upper-semicontinuous; this is particularly true if $(X,f)$ is expansive (see~\cite{Walters}; see also~\cite{Velozo} for others examples of systems for which the entropy function is upper-semicontinuous, including a system where $X$ is non-compact). The point is that the main strategy, up to this work, used in the proof that $\mathcal{E}_0(f)$ is a $G_\delta$ subset of $\M(f)$ consists in showing that the entropy map is upper-semicontinuous (see~\cite{Hofbauer,Oprocha} as examples of this fact); our strategy applies whether or no this is true. 

  As a matter of fact, it follows from Corollary~34 in~\cite{Catsigeras2017} that the metric entropy  of a generic continuous interval map $f:[0,1]\rightarrow[0,1]$ is not an upper-semicontinuous function of $\M(f)$ (Burguet has proven in~\cite{Burguet} that the metric entropy of such maps is not an upper-semicontinuous function of $\M_e(f)$). 
  This shows that there exist numerous examples of systems for which the entropy function is not upper-semicontinuous (in fact, as it was shown by these works, this can be quite rare for some families of maps).
\end{rek}

As a consequence of Corollary~\ref{corcentral1}, if $\M_p(f)$ is dense in $\M(f)$ (see \cite{Abdenur,Gelfert,Sigmund1970,Sigmund1971,Sigmund1974} for instances of systems which satisfy such condition) or even the weaker condition that $\mathcal{E}_0(f)$ is dense in $\M(f)$ (which is true if $(X,f)$ is a topological dynamical system which is transitive and has the shadowing property; see Proposition~3.8 in \cite{Oprocha} for details), one has the following result.

\begin{teo}
\label{central2}
Let $X$ be a Polish space, let $f:X\to X$ be a uniformly continuous function, and suppose that $\mathcal{E}_0(f)$ is dense in $\M(f)$. Then, $\mathcal{E}_0(f)$ is a dense $G_\delta$ subset of $\M(f)$.
\end{teo}

The fact that $\M_p(f)$ is a dense subset of $\M(f)$ is particularly true if $(X,f)$ satisfies the periodic specification property (see~\cite{Sigmund1974}); $(X,f)$ has the periodic specification property if, for each $\ve>0$, there exists $M(\ve)\in\N$ such that for each $x_1,\ldots x_k\in X$, for each choice of intervals of integers $A_1=[a_1,b_1],\ldots,A_k=[a_k,b_k]$ such that, for each $i=2,\ldots,k$, $a_i-b_{i-1}>M(\ve)$, and for each integer $p>b_k-a_1+M(\ve)$, there exists an $f$-periodic point, $x\in X$, of period $p$ such that, for each $i=1,\ldots,k$ and each $j\in A_i$, $d(f^jx,f^jx_i)<\ve$ (this definition does not depend on the choice of the metric $d$).


Thus, the following result is a particular case of Theorem~\ref{central2} and settles a conjecture posed by Sigmund in~\cite{Sigmund1974} (we remark that this conjecture was partially settled by Theorem~5.2 in~\cite{AS2}, in the particular case that $f$ is a Lipschitz map).

\begin{teo}\label{sigmund}
  Let $(X,f)$ be topological dynamical system and suppose that $(X,f)$ satisfies the periodic specification property. Then, $\mathcal{E}_0(f)$ is a dense $G_\delta$ subset of $\M(f)$.
\end{teo}

\begin{rek} Since a continuous map $f:[0,1]\rightarrow[0,1]$ satisfies the periodic specification property if and only if it is topologically mixing (see Theorem~8.7 in~\cite{Blokh1995}), Theorem~\ref{sigmund} settles a question posed by Blokh (namely, item 1) in Problem of Subsection~1.10 in~\cite{Blokh1995}, who asked whether or no $\mathcal{E}_0(f)$ is a residual subset of $\M(f)$ if $f$ is a topologically mixing map.
\end{rek}  

In fact, there are several other situations where it is known that $\mathcal{E}_0(f)$ is a generic subset of $\M(f)$ (that is, it contains a dense $G_\delta$ subset). Namely, this is the case when:
\begin{enumerate}
\item $(X,f)$ is an expansive topological dynamical system 
  that satisfies the periodic specification property (see Remark~\ref{semicontinuity} and the paragraph after Theorem~\ref{central2}); this is true for Axiom-A systems, full-shifts over finite alphabets and for arbitrary topologically mixing expanding maps over compact spaces (\cite{Sigmundlibro,Sigmund1974,Viana}). Recall that a homeomorphism $f:X\rightarrow X$ is expansive if there exists a constant $\ve> 0$ such that for any $x\neq y\in X$, there exists an $n\in\Z$ such that $d(f^nx,f^ny)\ge\ve$; see Subsection~\ref{expmaps} for a more detailed discussion involving expanding maps;
\item $(X,f)$ is transitive, has the shadowing property, and $\M(f)\ni\mu\mapsto h_\mu(f)\in[0,h_{top}(f)]$ is an upper-semicontinuous function (see Proposition~3.8 in \cite{Oprocha}); in fact, Theorem~\ref{central2} extends this result to the more general setting where the metric entropy is not upper-semicontinuous; 
\item $X=[0,1]$ and $f$ is a continuous piecewise monotonic function (\cite{Blokh1995,Hofbauer});
\item $(X,f)$ is such that $f:M\rightarrow M$ is a $C^1$-generic diffeomorphism over a compact boundaryless manifold $M$ and $X$ is an isolated non-trivial transitive set of $f$ (\cite{Abdenur}).
\end{enumerate}

For examples of applications of Theorem~\ref{central2} to systems that do not (necessarily) satisfy the periodic specification property and for which the entropy function is not (necessarily) upper-semicontinuous, one may highlight the following situations:
\begin{enumerate}
\item $X$ is a Polish metric space and $f$ is a continuous function for which there exists $K\subset\mathrm{Per}(f)$ such that $(X,f)$ is $K$-closable (see~\cite{Gelfert} for the definition);  this is a consequence of Theorem~4.11 in~\cite{Gelfert}. This is particularly true for $S$-gap shifts and $\beta$-shifts, as defined in~\cite{Gelfert}; 
\item as mentioned before, $(X,f)$ is a topological dynamical system which is transitive and has the shadowing property; this is a consequence of the first part of Proposition~3.8 in~\cite{Oprocha} (see~\cite{Oprocha} for examples of systems satisfying such conditions).
\item $X$ is a compact Riemannian manifold and $f$ is a mixing $C^{1+\alpha}$ ($\alpha>0$) diffeomorphism which preserves a hyperbolic Borel probability measure $\mu$; it was showed in~\cite{Hirayama} that there exists a Borel set $\tilde{\Lambda}$ of full $\mu$-measure such that $\M_p(f)$ is dense in the set of all $f$-invariant probability measures supported on $\tilde{\Lambda}$. 
\end{enumerate}

Let $f:[0,1]\rightarrow[0,1]$ be a continuous function. It follows from Corollary~10.5 in \cite{Blokh1995} that $\mathcal{E}_0^e(f)$ is dense in $\M_e(f)$, and therefore, by Theorem~\ref{central1}, that $\mathcal{E}_0^e(f)$ is a dense $G_\delta$ subset of $\M_e(f)$. However, for a typical continuous function $f\in\mathcal{C}([0,1])$, $\M_e(f)$ is a meager subset of $\M(f)$ (this is a consequence of Theorems~16 and~24 in~\cite{Catsigeras2017} and  Theorems~1 and~2 in~\cite{Catsigeras2019I}), so it is not true, for such functions, that $\mathcal{E}_0(f)$ is a dense $G_\delta$ subset of $\M(f)$.

More generally, Catsigeras has shown in Theorem 2 in~\cite{Catsigeras2019} that for a typical map $f\in C(M)$, where $M$ is a compact manifold of dimension $m$,  there exists a sequence of ergodic measures $\mu_n$, with $h_{\mu_n}(f)=\infty$, such that $\mu_n\to\mu$ and $h_\mu(f)=0$; therefore, its entropy function is not upper-semicontinuous. 
Furthermore, for such $f\in C(M)$, it follows from Corollary~15 in \cite{Catsigeras2019I} that the set $\M_p(f)$ is dense in $\M_e(f)$; hence, by Theorem~\ref{central1}, $\mathcal{E}_0^e(f)$ is a dense $G_\delta$ subset of $\M_e(f)$, although 
by  Theorem~1 in~\cite{Catsigeras2019I}, $\mathcal{E}_0^e(f)$ is a meager subset of $\M(f)$. 


\subsection{Correlation entropies}
We also have some results regarding the so-called correlation entropies (they are an entropy analogue of the Hentschel-Procaccia spectrum of generalized dimensions; see Chapter 2 in~\cite{Verbitskiy} for a broader discussion).

Let $(X,d)$ be a Polish metric space and let $f: X \to X$ be a continuous transformation preserving a Borel probability measure $\mu$. Set, for each $n\in\N$ and each $x,y\in X$,
\[d_n(x,y):=\max\{d(f^ix,f^iy)\mid 0\le i< n\},
\]
and set, for each $\ve>0$,
\[B_n(x,\ve):=\{y\in X\mid d_n(x,y)<\ve\},
\]
the so-called (open) Bowen ball of size $n$ and radius $\ve$ centered at $x\in X$.

For each $q \in \R$, $q\neq 1$, one defines the lower and the upper correlation entropies of order $q$ as follows: 
\begin{eqnarray*}
\underline{H}_{\mu}(f, q)&=&\lim _{\varepsilon \rightarrow 0} \liminf _{n \rightarrow \infty}-\frac{1}{(q-1) n} \log \int_{\supp(\mu)} \mu\left({B}_{n}(x, \varepsilon)\right)^{q-1} d \mu(x), \\ 
\overline{H}_{\mu}(f, q)&=&\lim _{\varepsilon \rightarrow 0} \limsup _{n \rightarrow \infty}-\frac{1}{(q-1) n} \log \int_{\supp(\mu)} \mu\left({B}_{n}(x, \varepsilon)\right)^{q-1} d \mu(x);
\end{eqnarray*}
for $q = 1$, one defines
\begin{eqnarray*}
\underline{H}_{\mu}(f, 1)&=&\lim _{\varepsilon \rightarrow 0} \liminf _{n \rightarrow \infty}-\frac{1}{ n} \int_{\supp(\mu)}  \log\mu\left({B}_{n}(x, \varepsilon)\right) d \mu(x), \\ 
\overline{H}_{\mu}(f, 1)&=&\lim _{\varepsilon \rightarrow 0} \limsup _{n \rightarrow \infty}-\frac{1}{ n} \int_{\supp(\mu)} \log \mu\left({B}_{n}(x, \varepsilon)\right) d \mu(x)
\end{eqnarray*}
(since $f$ is a continuous function, if $x\in\supp(\mu)$, then for each $n\in\N$ and each $\ve>0$, $\mu(B_n(x,\ve))>0$).

Let us recall some properties of the upper and the lower correlation entropies.

\begin{propo}\label{Verbitskiy} Let $(X,d)$ be a Polish metric space, let $f: X \to X$ be a continuous function, and let $\mu\in\M(f)$.
\begin{itemize}
\item[i)] For each $q_{1}<q_{2}$, $q_1,q_2\neq 1$, one has
\[\underline{H}_{\mu}\left(f, q_{1}\right) \geq \underline{H}_{\mu}\left(f, q_{2}\right) \geq 0, \quad \overline{H}_{\mu}\left(f, q_{1}\right) \geq \overline{H}_{\mu}\left(f, q_{2}\right) \geq 0;\]
if $h_\mu(f)<\infty$, then for each $q_{1}<1<q_{2}$, one has
\[\underline{H}_{\mu}\left(f, q_{1}\right) \geq \underline{H}_{\mu}\left(f, 1\right) \geq \underline{H}_{\mu}\left(f, q_2\right), \quad \overline{H}_{\mu}\left(f, q_{1}\right) \geq \overline{H}_{\mu}\left(f, 1\right) \geq  \overline{H}_{\mu}\left(f, q_2\right);\]
\item[ii)] $\overline{H}_{\mu}(f, 1)\le h_{\mu}(f)$, and if $\mu\in\M_e(f)$, then   $h_{\mu}(f)\le \underline{H}_{\mu}(f, 1)$;
\item[iii)] for $q > 1$, one has $$\underline{H}_{\mu}\left(f, q\right) \leq \overline{H}_{\mu}\left(f, q\right)\leq h_\mu(f);$$
\item[iv)] if $\mu\in\M_e(f)$, then for $q \in [0,1]$, one has $$h_\mu(f)\leq \underline{H}_{\mu}\left(f, q\right) \leq \overline{H}_{\mu}\left(f, q\right)\leq h(f),$$
where $h(f)$ is the topological entropy of $f$ (which can be infinite);
\item[v)] $\underline{H}_{\mu}\left(f, q\right)$ and $\overline{H}_{\mu}\left(f, q\right)$ depend continuously on $q$ for $q \in (q^\ast, 1)\cup (1,\infty)$, where $q^\ast:=\inf\{q\in\mathbb{R}\mid \overline{H}_\mu(f,q)<\infty\}$.
\end{itemize}
\end{propo}
\begin{proof}
\begin{itemize}
\item[i)] The first pair of inequalities follow from Lemma~2.1 in~\cite{Verbitskiy}. Now, it follows from the proof of Lemma~2.14 in~\cite{Verbitskiy} that if $h_\mu(f)<\infty$, then $-\frac{1}{n}\int_{\supp(\mu)}\log\mu(B_n(x,\ve))d\mu<\infty$. The second pair of inequalities is now a consequence of Lemma~2.13 in~\cite{Verbitskiy}.
\item[ii)] The proof of the first inequality is presented in the proof of Lemma~2.14 in~\cite{Verbitskiy}. The second inequality combines Theorem~\ref{BrinKatok} with the fact that
  \[\underline{h}_\mu^{loc}(f)\le\lim_{\ve\to 0}\liminf_{n\to\infty}-\frac{1}{n}\int_{\supp(\mu)}\log\mu(B_n(x,\ve))d\mu\]
  (which is a consequence of Fatou's Lemma).   
\item[iii)] The proof combines items i) and ii).
\item[iv)] It is the same as item iii).
\item[v)] These are Lemma~2.26 and Theorem~2.27 in~\cite{Verbitskiy}.
\end{itemize}  
\end{proof}
If there exists $q \in \mathbb{R}$ such that $\underline{H}_{\mu}(f, q)=\overline{H}_{\mu}(f, q)$, one says that the correlation entropy $H_{\mu}(f, q)$ exists; in this case, one sets
\[H_{\mu}(f, q):=\underline{H}_{\mu}(f, q)=\overline{H}_{\mu}(f, q).\]

A direct consequence of Theorem~\ref{central2} and items ii), iii) of Proposition~\ref{Verbitskiy} is the following result.

\begin{cor}
Let $X$ be a  Polish space, let $f:X\to X$ be a uniformly continuous function and suppose that $\mathcal{E}_0(f)$ is a dense subset of $\M(f)$. Then, the set $\mathcal{CE}_0=\{\mu\in \M(f)\mid H_{\mu}\left(f, q\right)=0,\,\forall q\ge 1\}$ is residual in $\M(f)$.
\end{cor}

The next results show that if $X$ is compact and $f:X\rightarrow X$ is a positively expansive or a Lipshitz function, the correlation entropies for $q<1$ are typically equal to zero. Recall that a continuous function $f:X\rightarrow X$ is positively expansive if there exists a constant $\ve> 0$ such that for any $x\neq y\in X$, there exists an $n\in\N$ such that $d(f^nx,f^ny)\ge\ve$.

\begin{teo}\label{entrinf1}
Let $X$ be a compact metric space, let $f:X\rightarrow X$ be a  positively expansive function and suppose that $\M_p(f)$ is a dense subset of $\M(f)$. Then, the set $\mathcal{H}_-=\{\mu\in \M(f)\mid \underline{H}_{\mu}\left(f, s\right)=0,\,\forall s\in(0,1)\}$ is residual in $\M(f)$.
\end{teo}


We remark that in Theorem~\ref{entrinf1}, $f:X\rightarrow X$ may be taken as an expansive function (just set $d_n(x,y):=\max\{d(f^ix,f^iy)\mid -n\le i\le n\}$ in the proof of this theorem). Thus, the result stated in Theorem~\ref{entrinf1} is particularly true for expansive maps that satisfy the periodic specification property (such as Axiom-A systems, full-shifts over finite alphabets and topologically mixing expanding maps over compact spaces, as mentioned before; recall that they cannot be positively expansive, otherwise $X$ would be finite~\cite{Mai-Sun}).


\begin{teo}\label{entrinf2}
Let $X$ be a compact metric space, let $f:X\to X$ be a Lipshitz function and suppose that $\M_p(f)$ is a dense subset of $\M(f)$. Then, the set $\mathcal{H}_-=\{\mu\in \M(f)\mid \underline{H}_{\mu}\left(f, s\right)=0,\,\forall s\in(0,1)\}$ is residual in $\M(f)$.
\end{teo}

There are numerous examples in the literature of dynamical systems that satisfy the hypotheses of Theorem~\ref{entrinf2}. This is particularly true for full-shifts $(X,\sigma)$ over compact metric spaces $(M,d)$, where $X=\prod_{i\in\Z}M$ is endowed with the product metric $\rho:X\times X\rightarrow[0,3]$, $\rho(x,y)=\sum_{i\in\Z}\frac{1}{2^{|i|}}\frac{d(x_i,y_i)}{1+d(x_i,y_i)}$ (such systems satisfy the periodic specification property, so $\M_p(\sigma)$ is a dense subset of $\M(\sigma)$; see~\cite{AS,Sigmund1974}). This is also true for Axiom-A systems (see \cite{Sigmund1970}) and for the case where $T:M\rightarrow M$ is a $C^1$-generic diffeomorphism over a compact boundaryless manifold $M$, $X$ is an isolated non-trivial transitive set of $T$ and $f:=T\restriction X$ (see \cite{Abdenur}), where in both cases, we endow the manifolds with their geodesic metrics.

\subsection{Packing dimension, rate of recurrence and quantitative waiting time}
\label{expmaps}

We also have something to say about the typical values of the packing dimension of an invariant measure, the rate of recurrence and the quantitative waiting time of (local) expanding maps (see~\cite{AS} and references therein for the motivations behind such definitions).

\begin{defi} \label{expanding} A map $f: X\rightarrow X$ on the metric space $(X,d)$ is called (locally) expanding if there exist $\lambda>1$ and $\ve_0>0$ such that, for each  $x,y\in X$ satisfying $d(x,y)<\ve_0$, one has
  \[d(fx,fy)\ge\lambda d(x,y).\]
\end{defi}

It is clear that if $f$ is expanding, then it is positively expansive. The next result shows that, for compact metric spaces, the converse is somewhat true. 

\begin{teo} [\cite{Fathi1989,Reddy}] Let $(X,d)$ be a compact metric space and let $f:X\rightarrow X$ be a continuous positively expansive map. Then, there exists a metric $d^\prime$ on $X$, compatible with $d$ (i.e., the topologies generated on $X$ by d and $d^\prime$ are the same), such that $f$ is expanding on $(X,d^\prime)$.
\label{expexp}  
\end{teo}

\begin{defi}[radius packing $\phi$-premeasure,~\cite{Cutler1995}] Let $\emptyset\neq E\subset X$, and let $0 < \delta < 1$. A $\delta\textrm{-}$\emph{packing} of $E$ is a countable collection of disjoint closed balls $\{B(x_k,r_k)\}_k$ with centers $x_k\in E$ and radii satisfying  $0<r_k\leq \delta/2$, for each $k\in\N$ (the centers $x_k$ and radii $r_k$ are considered part of the definition of the $\delta$-packing). Given a measurable function $\phi$, the radius packing ($\phi,\delta$)\textrm{-}premeasure of $E$ is given by the law
\begin{eqnarray*}
P^{\phi}_{\delta}(E)=\sup\left\{\sum _{k=1}^{\infty} \phi(2r_k)\mid \{B(x_k,r_k)\}_k \mbox{ is a } \delta\textrm{-}\mbox{packing} \mbox{ of E}\right\}.
\end{eqnarray*}
Letting $\delta\to 0$, one gets the so-called \emph{radius packing} $\phi\textrm{-}$\emph{premeasure}
\begin{eqnarray*}
P^{\phi}_{0}(E)=\lim_{\delta \to 0} P^{\phi}_{\delta}(E).
\end{eqnarray*}
One sets $P^{\phi}_{\delta}(\emptyset)=P^{\phi}_{0}(\emptyset)=0$.
%
Finally, the radius packing $\phi\textrm{-}$measure of $E\subset X$ is defined to be 
\begin{eqnarray*}
\label{pakmeasure}
P^{\phi}(E)=\inf\left\{ \sum_{k}P^{\phi}_{0}(E_k)\mid E\subset \bigcup_k E_k \right\},
\end{eqnarray*}
\noindent where the infimum is taken over all countable coverings $\{E_k\}_k$ of $E$. 
\end{defi}

Of special interest is the situation where given $\alpha>0$, one sets $\phi(t)=t^{\alpha}$. In this case, one uses the notation $P^{\alpha}_0$ and refers to $P^{\alpha}_0(E)$ as the $\alpha\textrm{-}$packing premeasure of $E$. We note that $\dim_P(X)$ may be infinite. 

\begin{defi}[lower and upper packing dimensions of a measure,~\cite{Mattila}]\label{HPdim}
  Let $\mu$ be a positive Borel measure on $(X,\mathcal{B})$. 
  The lower and upper packing dimensions of $\mu$ are defined, respectively, by 
\begin{eqnarray*}
\dim_{P}^-(\mu)&=&\inf\{\dim_{P}(E)\mid \mu( E)>0, ~ E\in \mathcal{B}\},\\
\dim_{P}^+(\mu)&=&\inf\{\dim_{P}(E)\mid \mu(X\setminus E)=0, ~ E\in \mathcal{B}\}.
\end{eqnarray*}
If $\dim_{P}^-(\mu)=\dim_{P}^+(\mu)$, one denotes the common value by $\dim_{P}(\mu)$. 
\end{defi}

\begin{defi} [lower and upper recurrence rates of $x\in X$,~\cite{Barreira2001}]Let $X$ be a separable metric space and let $f:X\rightarrow X$ be a Borel measurable transformation. Let also, for each $x\in X$ and each $r>0$, 
\begin{eqnarray*}
\tau_r(x)=\inf\{k\in\N\mid f^k x \in {B}(x,r)\}
\end{eqnarray*}
be the return time of a point $x$ into the ball ${B}(x, r)$ (note that $\tau_r(x)$ may be infinite on a set of zero $\mu$-measure). Then, 
\begin{eqnarray*}
\underline{R}(x)=\liminf_{r\to 0}\frac{\log \tau_r(x)}{-\log r} ~~\mbox{ and } ~~  \overline{R}(x)=\limsup_{r\to 0}\frac{\log \tau_r(x)}{-\log r}
\end{eqnarray*}
are, respectively, the lower and upper recurrence rates of $x$. If $\underline{R}(x)=\overline{R}(x)$, one denotes the common value by $R(x)$. 
\end{defi}

\begin{defi}[upper and lower quantitative waiting time indicators,~\cite{Galatolo}] Let $X$ be a separable metric space and let $f:X\rightarrow X$ be a Borel measurable transformation. Let also $x,y\in X$ and $r>0$. The first entrance time of $\mathcal{O}(x):=\{f^ix\mid i\in\mathbb{Z}\}$, the $f$-orbit of $x$, into the ball ${B}(y,r)$ is given by
\begin{eqnarray*}
\tau_r(x,y)=\inf\{n\in \N \mid f^nx\in {B}(y,r) \}
\end{eqnarray*} 
(note that $\tau_r(x,y)$ may be infinite on a set of zero $\mu\times\mu$-measure). The so-called quantitative waiting time indicators are defined as
\begin{eqnarray*}
\underline{R}(x,y)=\liminf_{r\to 0}\frac{\log \tau_r(x,y)}{-\log r} ~~\mbox{ and } ~~  
\overline{R}(x,y)=\limsup_{r\to 0}\frac{\log \tau_r(x,y)}{-\log r}.
\end{eqnarray*}
If $\underline{R}(x,y)=\overline{R}(x,y)$, one denotes the common value by $R(x,y)$.
\end{defi}

The next theorem contains results from Proposition~A in~\cite{Varandas} and Lemma~2.1 in~\cite{AS2} (combined with Brin-Katok's Theorem).

\begin{teo}\label{Varandas} Let $X$ be a compact metric space and let $f:X\rightarrow X$ be a continuous expanding function, with expanding constant $\lambda>1$. If $\mu\in\M_e(f)$, then
  \begin{enumerate}
  \item $\dim_P^+(\mu)\le\dfrac{h_\mu(f)}{\log\lambda}$;
  \item $\overline{R}(x)\le\dfrac{h_\mu(f)}{\log\lambda}$, for $\mu$-a.e $x\in X$;     
  \item $\overline{R}(x,y)\le\dfrac{h_\mu(f)}{\log\lambda}$, for $(\mu\times\mu)$-a.e $(x,y)\in X\times X$.
  \end{enumerate}
\end{teo}

\begin{rek} The proof of item~3 in~\ref{Varandas} follows from an almost verbatim adaptation of the proof of item~1; in fact, in order to obtain the result, one just needs to prove that for $(\mu\times\mu)$-a.e $(x,y)\in X\times X$, $\lim_{\ve\to 0}\limsup_{n\to\infty}\frac{1}{n}R_n(x,y,\ve)\le h_\mu(f)$, where $R_n(x,y,\ve):=\inf\{k\ge 1\mid f^kx\in B_n(y,n,\ve)\}$, which follows from the arguments presented in the proof of Theorem~A in~\cite{Varandas}. 
\end{rek}

One may combine Theorems~\ref{central2} and~\ref{Varandas} in order to prove the following result.

\begin{teo}\label{all}
  Let $X$ be a compact metric space, let $f:X\to X$ be a continuous expanding function, and suppose that $\mathcal{E}_0(f)$ is dense in $\M(f)$.
  Then, each of the sets $PD:=\{\mu\in\M(f)\mid\dim_P(\mu)=0\}$, $\mathcal{R}:=\{\mu\in\M(f)\mid R(x)=0,$ for $\mu\textrm{-}a.e.\,x\}$ and $\mathscr{R}:=\{\mu\in\M(f)\mid R(x,y)=0,$ for $(\mu\times\mu)\textrm{-}a.e.\,(x,y)\}$ is a residual subset of $\M(f)$.
\end{teo}

\begin{rek} It follows from Theorem~\ref{expexp} that the results stated in Theorem~\ref{all} are valid for positively expansive maps if we endow $X$ with the hyperbolic metric $d^\prime$.
\end{rek}  

Theorem~\ref{all} states that if $f$ is a continuous expanding function defined on $X$ such that $\mathcal{E}_0(f)$ is dense in $\M(f)$ (it suffices that $\M_p(f)$ is dense in $\M(f)$),
then, for each $\mu$ in the residual set $PD\cap\mathcal{R}\cap\mathscr{R}$, $d_\mu(x)=\dim_P(\mu)=R(x)=R(x,y)=0$ for $(\mu\times\mu)$-a.e. $(x,y)\in X\times X$, where $d_\mu(x):=\lim_{r\to 0}\frac{\log\mu(B(x,r))}{\log r}$ stands for the local dimension of $\mu$ at $x\in X$ (see Proposition~1.1 in~\cite{AS}). In other words, a typical measure of such systems behaves (in terms of these indicators) as a periodic measure.
 
Important examples of expanding maps on compact sets for which $\M_p(f)$ is a dense subset of $\M(f)$ are:

\begin{enumerate}
\item topologically exact expanding maps ($f$ is topologically exact if for each non-empty open set $U\subset X$, there exists $N\in\N$ such that $f^N(U)=X$; every topologically mixing expanding map is topologically exact), since they satisfy the periodic specification property; see~Theorem~11.3.1 in~\cite{Viana};
\item any expanding map defined over a connected and compact (continuum) metric space; namely, it follows from Corollary~11.2.16 in~\cite{Viana} that in this case, the expanding map is topologically exact. 
\end{enumerate}

\subsection{Expansive measures}

The method used in proof of Theorem~\ref{central1} can be used to show that the set of expansive measures is a $G_{\delta\sigma}$ set in the set of probability measures $\M(X)$, where $X$ is a Polish metric space (this result was proved by Lee, Morales and Shin in \cite{Morales2018} for compact spaces; we obtain an alternative proof of this fact). 

\begin{defi}
Let $f:X \rightarrow X$ be a homeomorphism on the metric space $X$. A Borel measure $\mu$ is said to be an expansive measure for $f$ if there exists $\delta > 0$ such that for each $x\in X$, $\mu(\Gamma_\delta(x)) = 0$, where $\Gamma_{\delta}(x):=\left\{y \in X\mid d\left(f^{i}(x), f^{i}(y)\right) \leq \delta, \forall i \in \mathbb{Z}\right\}$. The constant $\delta $ is the so-called expansivity constant of $\mu$. The set of expansive measures of $f$ is denoted by $\M_{exp}$.
\end{defi}

\begin{teo}
\label{GExpansive}
Let $f: X \rightarrow X$ be a uniform homeomorphism of a Polish metric space $X $. Then, $\M_{exp}$ is a $G_{\delta\sigma}$ subset of $\M(X)$.
\end{teo}




\subsection{Organization}

The paper is organized as follows. In Section~\ref{Zeroentropy} we present, among other important results, the proof of Theorem~\ref{central1}. In Section~\ref{SCE} we present the proofs of Theorems~\ref{entrinf1} and~\ref{entrinf2}, and finally, in Section~\ref{expasivemeasures}, we present the proof of Theorem~\ref{GExpansive}.

\section{Zero entropy is generic}
\label{Zeroentropy}

In this section, we prove Theorem~\ref{central1} using a local representation of the metric entropy of an ergodic measure (this is Theorem~\ref{BrinKatok}, a partial extension of Brin-Katok's Theorem) and showing that the set of ergodic measures with lower local entropy less or equal to $\alpha\in(0,\infty)$ is a $G_\delta$ subset of $\M_e(f)$ (this is Lemma~\ref{Gdelta3}).


\begin{lema}
  \label{asympt1}
  Let $X$ be a Polish space endowed with the metric $d$, let $f:X\rightarrow X$ be a uniformly continuous function, and let $\mu\in\M(X)$. Let also, for each $x\in X$, each $\ve>0$ and each $n\in\N$, $g_{x,\ve,n}(\,\cdot\,):\M(X)\rightarrow [0,1]$ be defined by the law
\[g_{x,\ve,n}(\mu):=\int g^{\ve,n}_x(y)d\mu(y),\]
where  $g^{\ve,n}_x:X\rightarrow[0,1]$ is defined by the law 
\[
g^{\ve,n}_x(y):= \left\{ \begin{array}{lcc}
                    
            1 & ,if  & d_n(x,y) \leq \ve, \\
            \\ -\dfrac{d_n(x,y)}{\ve}+2&, if & \ve\le d_n(x,y)\le 2\ve,\\
            \\ 0 &   ,if  & d_n(x,y) \geq 2\ve.
          \end{array}
\right.\]
Then, the function $g_{\ve,n}(\cdot,\cdot):\M(X)\times X\rightarrow [0,1]$, $g_{\ve,n}(\mu,x)=g_{x,\ve,n}(\mu)$, is jointly continuous. Furthermore, for each $x\in X$, each $\ve>0$ and each $n\in\N$, one has
\[\mu(B_n(x,\ve))\le g_{x,\ve,n}(\mu)\leq \mu(B_n(x,2\ve)).\] 
\end{lema}
\begin{proof}
 In order to prove the second assertion, just note that, for each $x\in X$, each $\ve>0$ and each $n\in\N$, $g^{\ve,n}_x: X\rightarrow \R$ is a continuous function such that, for each $y\in X$, $ \chi_{_{B_n(x,\ve)}}(y) \leq g_{x}^{\ve,n}(y)\leq \chi_{_{B_n(x,2\ve)}}(y)$.

  Now, given that $g_x^{\ve,n}(y)$ only depends on $d_n(x,y)$, it is straightforward to show that $g^{\ve,n}_{x_m}$ converges uniformly to $g^{\ve,n}_x$ on $X$ when $d(x_m,x)\rightarrow 0$ (since each of the functions $f,\ldots,f^{n-1}$ is uniformly continuous, for each $\eta>0$, there exists $\delta>0$ so that if $d(w,x)<\delta$, then $d_n(w,x)<\eta$).

The proof of the first assertion now follows from the same arguments presented in the proof of Lemma~2.1 in~\cite{AS}.
\end{proof}

\begin{rek} The results stated in Lemma~\ref{asympt1} are particularly true if $(X,T)$ is a topological dynamical system, since in this case, $f$ is uniformly continuous.
\end{rek}


Since, for each $x\in X$,
\[\underline{h}_{\mu}(f,x):=\lim_{\ve\to 0}\liminf_{n\to\infty}\frac{\log \mu(B_n(x,\ve))}{-n}=\lim_{\ve\to 0}\lim_{s\to \infty}\inf_{n\geq s} \frac{\log \mu(B_n(x,\ve))}{-n},\]
we set, for each $x\in X$, each $\ve>0$ and each $s\in\mathbb{N}$, 
\begin{eqnarray*}
  \underline{\beta}_{\mu}^\ve(x,s)=\inf_{n> s} \frac{\log \mu(B_n(x,\ve))}{-n};
\end{eqnarray*}
note that, for each $\ve>0$, $\N\ni s\mapsto\underline{\beta}_{\mu}^\ve(x,s)\in[0,\infty]$ is a non-decreasing function, whereas, for each $s\in\N$, $\mathbb{R}_+\ni \ve\mapsto\underline{\beta}_{\mu}^\ve(x,s)\in[0,\infty]$ is a non-increasing function.

Let, for each $x\in X$, each $\ve>0$ and each $s\in\N$,
\[{\underline{\gamma}}_\mu^\ve(x,s):=\inf_{n\ge s}\frac{\log(g_{x,\ve,n}(\mu))}{-n},\]
where $g_{x,\ve,n}(\mu)$ is defined as in Lemma~\ref{asympt1}. 

\begin{lema}
\label{Gdelta3} 
Let $X$ be a Polish metric space and let $f:X\rightarrow X$ be a continuous function. Then, for each $\alpha>0$, each $\ve>0$ and each $s\in\N$, 
 \[\underline{M}(\alpha,\ve,s):=\{ \mu\in \M(X)\mid\mu\textrm{-}\esssup\underline{\gamma}_{\mu}^\ve(x,s)\le\alpha\}\] 
 is a $G_{\delta}$ subset of $\M(X)$.
\end{lema}

\begin{proof}
  This is basically the proof of Proposition~2.1 in~\cite{AS}; we present the details for the reader's sake.

  Fix $\alpha>0$, $\ve>0$ and $s\in\N$. Note that it is enough to prove that $\M(X)\setminus \underline{M}(\alpha,\ve,s)=\{\mu\in \M(X) \mid\mu(\{x\in X\mid\underline{\gamma}_{\mu}^\ve(x,s)>\alpha\})>0\}$ is an $F_{\sigma}$ set. 

Let $\mu \in \M(X)$, let $k\in \N$, set $Z_{\mu}(k):=\{x\in X\mid \underline{\gamma}_{\mu}^\ve(x,s)\geq\alpha+1/k\}$ and set, for each $l\in\N$,
\begin{eqnarray*}
\M(k,l):=\{\nu\in \M(X)\mid \nu(Z_{\nu}(k))\ge1/l\}.
\end{eqnarray*}

\textit{Claim 1.} $Z_{\mu}(k)$ is  closed.

Let $(z_m)$ be a sequence in $Z_{\mu}(k)$ such that $z_m\to z$. Since,  by Lemma~\ref{asympt1}, for each $n\in\N$ and each $\mu\in\M(X)$, $X\ni x\mapsto g_{x,\ve,n}(\mu)\in[0,1]$ is a continuous function, the mapping $X\ni x\longmapsto  \underline{\gamma}_{\mu}^\ve(x,s)\in[0,\infty]$ is upper-semicontinuous, so $z\in Z_{\mu}(k)$. 

\textit{Claim 2.} $W_{k}=\{(\mu,x)\in \M(X)\times X\mid \underline{\gamma}_{\mu}^\ve(x,s)<\alpha+1/k\}$ is open.

It follows from Lemma~\ref{asympt1} that the mapping $\M(X)\times X\ni(\mu,x)\longmapsto \underline{\gamma}_{\mu}^\ve(x,s)$
is upper-semicontinuous. 

\

Now, we show that $\M(k,l)$ is closed. Let $(\mu_n)$ be a sequence in $\M(k,l)$ such that $\mu_n\to \mu$. Suppose, by absurd, that $\mu\notin \M(k,l)$; we will find that $\mu_n\notin \M(k,l)$ for~$n$ sufficiently large, a contradiction.

If $\mu\notin \M(k,l)$, then $\mu(A)>1-1/l$ where, $A=X\setminus Z_{\mu}(k)$.  Given that $\mu$ is tight ($\mu$ is a probability Borel measure and the space $X$ is Polish), there exists a compact set $C\subset A$ such that $\mu(C)>1-1/l$.

The idea is to construct a suitable subset of $W_k$ that contains a neighborhood of $\{\mu\}\times C$. Let, for each $x\in C$, $V_x\subset W_{k}$ be an open neighborhood of~$(\mu,x)$ (such open set exists, by Claim 2); that is, set $V_x:=B((\mu,x);\varepsilon)=\{(\nu,y)\in\M(X)\times X\mid\max\{\rho(\nu,\mu),d(x,y)\}<\varepsilon\}$, for some suitable $\varepsilon>0$ (where $\rho$ is any metric defined in $\M(X)$ which is compatible with the weak topology); then, $\{V_x\}_{x\in C}$ is an open cover of $\{\mu\}\times C$, and since $\{\mu\}\times C$ is a compact subset of $\M(X)\times X$, it follows that one can extract from $\{V_x\}_{x\in C}$ a finite subcover, $\{V_{x_i}\}_{i=1}^k$.

We affirm that there exists an $\ell\in\mathbb{N}$ (which depends on $C$) such that $\{\mu_n\}_{n\ge\ell}\subset\bigcap_{i}(\pi_1(V_{x_i}))$. Namely, for each $i$, there exists an $\ell_i$ such that $\{\mu_n\}_{n\ge\ell_i}\subset\pi_1(V_{x_i})$; set $\ell:=\max\{\ell_i\mid i\in\{1,\ldots,k\}\}$, and note that for each $i$, $\{\mu_n\}_{n\ge\ell}\subset\pi_1(V_{x_i})$. Set also $\mathcal{I}:=\bigcap_{i}(\pi_1(V_{x_i}))$ and $\mathcal{O}:=\bigcup_{i}(\pi_2(V_{x_i}))$.

Since for each $i$, $V_{x_i}=\pi_1(V_{x_i})\times\pi_2(V_{x_i})$, and given that
\begin{eqnarray*}
\{\mu_n\}_{n\ge\ell}\times\mathcal{O}&\subset& \mathcal{I}\times\mathcal{O}= \bigcup_{j} \left( \left[\bigcap_{i}\pi_1(V_{x_i}) \right] \times \pi_2(V_{x_j})\right)
\subset  \bigcup_{j}(\pi_1(V_{x_j})\times\pi_2(V_{x_j}))\\
&=&\bigcup_{j}V_{x_j}\subset W_{k},
\end{eqnarray*}

it follows that, for each $n\ge\ell$ and each $y\in\mathcal{O}$, $\underline{\gamma}_{\mu_n}^\ve(y,s)<\alpha+1/k$.

On the other hand, weak convergence implies
\[
\displaystyle\limsup_{n\to \infty}\mu_n(X\setminus\mathcal{O})\leq\mu(X\setminus\mathcal O)\leq  \mu(X\setminus C)=1-\mu(C)<\frac{1}{l},
\] from which it follows that there exists an $\tilde{\ell}\ge\ell$ such that, for each $n\ge\tilde{\ell}$, $\mu_n(X\setminus \mathcal{O})<1/l$.

Combining the last results, one concludes that for each $n\ge\tilde\ell$, $\mu_n(X\setminus \mathcal{O})<1/l$, and for each $x\in \mathcal{O}$, $\underline{\gamma}^\ve_{\mu_n}(y,s)<\alpha+1/k$, so
\[\mu_n(Z_{\mu_n}(k))\le\mu_n(X\setminus\mathcal{O})<1/l;\]
this contradicts the fact that, for each $n\in\mathbb{N}$, $\mu_n\in \mathcal{M}(k,l)$. Hence, $\mu \in  \mathcal{M}(k,l)$, and $\mathcal{M}(k,l)$ is a closed subset of $\mathcal{M}(f)$. 

Finally, it follows that $\M(X)\setminus\underline{M}(\alpha,\ve,s)=\bigcup_{k\in \N}\bigcup_{l\in \N}\M(k,l)$ is an $F_{\sigma}$ subset of $\M(X)$, and we are done. 
\end{proof}

\begin{teo}\label{BrinKatok} Let $X$ be a Polish space, let $f:X\rightarrow X$ be a measurable function, and let $\mu\in\M_e(f)$ (suppose that $\M_e(f)\neq\emptyset$). Then,
\begin{enumerate}
\item ${\underline{h}}_\mu(f,x)$ is $f$-invariant; 
\item $\int \underline{h}_\mu(f,x)d\mu(x)\le h_\mu(f)$ (with possibly $h_\mu(f)=\infty$).
\end{enumerate}
\end{teo}  
  \begin{proof}
 In order to prove item 1, we need the following result.

    {\textit{Claim.}} For each $x\in X$, one has ${\underline{h}}_\mu(f,fx)\le{\underline{h}}_\mu(f,x)$.

    Since, for each $x\in X$, each $\ve>0$ and each $n\in\N$, $f^{-1}(B_n(fx,\ve))\subset B_{n+1}(x,\ve)$, it follows from the monotonicity, for each $s\in\N$, of $(0,1]\ni\ve\mapsto{\underline{\beta}}_\mu^\ve(x,s)\in[0,\infty]$ that for each $x\in X$, both ${\underline{h}}_\mu(f,x)$ and ${\underline{h}}_\mu(f,fx)$ exist (they may be infinite) and satisfy the required inequalities.
    
We split the proof of item 1 into two cases.

{\textit{Case 1.}} $\int{\underline{h}}_\mu(f,x)d\mu(x)<\infty$. Since $\mu\in\M_e(f)$, one has $\int({\underline{h}}_\mu(f,x)-{\underline{h}}_\mu(f,fx))d\mu(x)=0$. The result follows now from Claim.

{\textit{Case 2.}} $\int{\underline{h}}_\mu(f,x)d\mu(x)=\infty$. Since $\mu\in\M(f)$, one has $\int{\underline{h}}_\mu(f,fx)d\mu(x)=\infty$ (otherwise, $\int{\underline{h}}_\mu(f,x)d\mu(x)<\infty$). Let $M>0$ and set $A_M^f:=\{x\in X\mid{\underline{h}}_\mu(f,fx)\ge M\}$; then, $\mu(A^f_M)>0$. 
It follows from Claim (replacing $x$ with $fx$) that $f^{-1}(A^f_M)\subset A^f_M$, and since $\mu\in\M_e(f)$, one has $\mu(A^f_M)=1$.  Hence, $A^f_\infty:=\bigcap_{M\in\N}A^f_M$ is such that $\mu(A^f_\infty)=1$.

It also follows from Claim that $\mu(A_\infty)=1$, where $A_\infty:=\{x\in X\mid {\underline{h}}_\mu(f,x)=\infty\}$. Summing up, $\mu(A_\infty\cap A^f_\infty)=1$, which means that ${\underline{h}}_\mu(f,fx)={\underline{h}}_\mu(f,x)=\infty$ for $\mu$-a.e. $x\in X$.

\

We proceed to the proof of item 2. Fix $\ve>0$ and let $\xi$ be an arbitrary countable measurable partition of $X$ such that $\diam(\xi):=\max_{\Delta\in\xi}\diam(\Delta)<\ve$; such partition exists, given that $X$ is separable. Then, for each $x\in X$, $\xi(x)\subset B(x,\ve)$, and hence, for each $n\in\mathbb{N}$,
    \[\xi^n(x)=\bigcap_{i=0}^{n-1}f^{-i}\xi(f^ix)\subset\bigcap_{i=0}^{n-1}f^{-i}B(f^ix,\ve)=B_n(x,\ve). \]

Therefore, for each $\ve>0$, it follows from Fatou's Lemma that
\begin{eqnarray*}\int\liminf_{n\to\infty}\frac{\log\mu(B_n(x,\ve))}{-n}d\mu(x)&\le&\int\liminf_{n\to\infty}\frac{\log\mu(\xi^n(x))}{-n}d\mu(x)\\
      &\le&\liminf_{n\to\infty}-\frac{1}{n}\left(\sum_{\Delta\in\xi^n}\mu(\Delta)\log\mu(\Delta)\right)\\
      &\le&\limsup_{n\to\infty}\frac{1}{n}H(\xi^n)\le h_\mu(f).
\end{eqnarray*}
Using again Fatou's Lemma, one gets
 \[\int\underline{h}_\mu(f,x)d\mu(x)\le h_\mu(f).\]
\end{proof}

  \begin{cor}  \label{BrinKatokC}
Let $X$ be a Polish space, let $f:X\rightarrow X$ be a continuous function, and let $\mu\in\M_e(f)$ (suppose that $\M_e(f)\neq\emptyset$). Then,
    \[\underline{h}_\mu^{loc}(f):=\mu\textrm{-}\essinf\underline{h}_\mu(f,x)=h_\mu(f).\]
\end{cor}
\begin{proof} We split the proof into two cases.

    {{\textit{Case 1.}}}  $h_\mu(f)<\infty$. Given that $\underline{h}_\mu^{loc}(f)\le\int \underline{h}_\mu(f,x) d\mu(x)$, inequality $\underline{h}_\mu^{loc}(f)\le h_\mu(f)$ is a direct consequence of item~2 of Theorem~\ref{BrinKatok}. Inequality $h_\mu(f)\le\underline{h}_\mu^{loc}(f) $ is just Theorem~2.9 in~\cite{Riquelme}.

 {{\textit{Case 2.}}} $h_\mu(f)=\infty$. The result follows from Theorem~2.9 in~\cite{Riquelme}.
\end{proof}

\begin{teo} 
\label{Gdeltaentropy} Let $X$ be a Polish space, let $f:X\rightarrow X$ be a uniformly continuous function, and let $\alpha\in(0,\infty)$. Then,
\[\mathcal{E}(\alpha):=\{\mu\in\M_e(f)\mid h_\mu(f)\le\alpha\}\]
is a $G_\delta$ subset of $\M_e(f)$. 
  \end{teo}
\begin{proof}
If $\M_e(f)=\emptyset$, there is nothing to prove; so, assume that $\M_e(f)\neq\emptyset$. We need the following assertions.

  {\textit{Claim 1.}} $\{\mu\in\M(f)\mid\mu\textrm{-}\esssup\underline{h}_\mu(f,x)\le\alpha\}=\bigcap_{k,s\ge 1}\{\mu\in\M(f)\mid\mu\textrm{-}\esssup\underline{\beta}_\mu^{1/k}(x,s)\le\alpha\}$.

  The first inclusion follows from the fact that, for each $x\in X$ and each $\mu\in\M(f)$, $\N\times\N\ni(k,s)\mapsto\underline{\beta}_\mu^{1/k}(x,s)\in[0,\infty]$ is a non-decreasing function on both variables.

  Now, let $\mu\in\bigcap_{k,s\ge 1}\{\mu\in\M_e(f)\mid\mu\textrm{-}\esssup\underline{\beta}_\mu^{1/k}(x,s)\le\alpha\}$. Then, for each $k,s\in\N$, there exits a measurable $A_{k,s}\subset X$, with $\mu(A_{k,s})=1$, such that for each $x\in A_{k,s}$, $\underline{\beta}_\mu^{1/k}(x,s)\le \alpha$. Let, for each $k\in\N$, $A_k:=\bigcap_{s\ge 1}A_{k,s}$, and then set $A:=\bigcap_{k\ge 1}A_{k}$; it follows that, for each $x\in A$, one has $\underline{h}_\mu(f,x)=\lim_{k\to\infty}\lim_{s\to\infty}\underline{\beta}_\mu^{1/k}(x,s)\le\alpha$. Since $\mu(A)=1$, the result follows.
  
\

{\textit{Claim 2.}} Let $\mu\in\M(f)$ and suppose that, for each $\ve>0$ and each $s\in\N$, $\mu\textrm{-}\esssup\underline{\beta}_\mu^\ve(x,s)\le\alpha$. Then, for each $\eta>0$ and each $p\in\N$, $\mu\textrm{-}\esssup\underline{\gamma}_\mu^\eta(x,p)\le\alpha$. The converse is also true. 

 The proof of the claim follows from the fact that, for each $x\in X$, each $\ve>0$, each $s\in\N$ and each $\mu\in\M(f)$, one has $\underline{\beta}_\mu^\ve(x,s)\le\underline{\gamma}_\mu^\ve(x,s)\le\underline{\beta}_\mu^{2\ve}(x,s)$ (see the proof of Lemma~\ref{asympt1}).

\

Since, for each $\mu\in\M_e(f)$, $\underline{h}_\mu^{loc}(f)=\mu\textrm{-}\esssup\underline{h}_\mu(f,x)$, it follows from 
Corollary~\ref{BrinKatokC} that 
\begin{eqnarray*}
    \mathcal{E}(\alpha)&=&\{\mu\in\M_e(f)\mid\mu\textrm{-}\esssup\underline{h}_\mu(f,x)\le\alpha\}=\bigcap_{k\ge 1}\bigcap_{s\ge 1}\{\mu\in\M_e(f)\mid\mu\textrm{-}\esssup\underline{\beta}_\mu^{1/k}(x,s)\le\alpha\}\\
    &=&\bigcap_{k\ge 1}\bigcap_{s\ge 1}\{\mu\in\M_e(f)\mid\mu\textrm{-}\esssup\underline{\gamma}_\mu^{1/k}(x,s)\le\alpha\},
  \end{eqnarray*}
  where we have used Claims 1 and 2 in the second and third equalities, respectively. 

The result follows now from Proposition~\ref{Gdelta3}.
\end{proof}

\begin{rek}\label{RGdeltaentropy} If $X$ is a compact metric space and $f:X\rightarrow X$ is a continuous function, 
  then the result stated in Theorem~\ref{Gdeltaentropy} remains valid. 
\end{rek}

\begin{proof1}
The result is a consequence of Theorem~\ref{Gdeltaentropy} and the fact that $\mathcal{E}_0^e(f)=\cap_{m\in\N}\mathcal{E}(1/m)$. 
\end{proof1}


\section{Correlation entropies}
\label{SCE}

In this section, we prove Theorems~\ref{entrinf1} and~\ref{entrinf2}. In what follows, $(X,f)$ is a topological dynamical system.

\subsection{Positively expansive maps}

Suppose also that $f$ is a positively expansive map. Let $\mu\in\M(X)$, let $s\in(0,1)$ and let $E=\{x_i\}$ be a finite $(n,r)$-generating set (which implies $X=\cup_{x_i\in E}B_n(x_i,r)$; given that $X$ is compact, such finite generating set exists). Let also $\tilde{E}=\{x_j\}$ be a subset of $E$ such that $\{B_n(x_j,r)\}_{\tilde{E}}$ is a covering of $\supp(\mu)$.

Since, for each $n\in\mathbb{N}$ and each $x\in B_n(x_j,r)$, one has $B_n(x_j,r)\subset B_n(x,2r)$, it follows that for each $x\in B_n(x_j,r)\cap\supp(\mu)$, $\mu(B_n(x_j,r))^{s-1}\ge \mu(B_n(x,2r))^{s-1}$; hence,
\begin{eqnarray}
\label{eq4}
I_{\mu}(s,2r,n)\nonumber
&:= & \int_{\supp(\mu)}\mu(B_n(x,2r))^{s-1}d\mu(x)\nonumber 
\leq  \sum_{x_j\in \tilde{E}}\int_{B_n(x_j,r)\cap\supp(\mu)} \mu(B_n(x,2r))^{s-1}d\mu(x)\nonumber \\
&\leq & \sum_{x_j\in \tilde{E}}\int_{B_n(x_j,r)\cap\supp(\mu)} \mu(B_n(x_j,r))^{s-1}d\mu(x)
=\nonumber \sum_{x_j\in \tilde{E}} \mu(B_n(x_j,r))^s \\
&\leq &\sum_{x_i\in E} \mu(B_n(x_i,r))^s.
\end{eqnarray}

Naturally, since $X$ is a compact metric space, one can assume, without loss of generality, that $E$ is always a finite $(n,r)$-generating set of $X$.

\begin{defi}
\label{sumacubos}
Let $\mu\in\M(X)$. One defines, for each $s\in (0,1)$, each $n\in\mathbb{N}$ and each $r>0$,
\[S_{\mu}(s,r,n)=\inf_{E}\sum_{x_j\in E} \mu(B_n(x_j,r))^s,\qquad W_{\mu}(s,r,n)=\inf_{E}\sum_{x_j\in E}g_{r,n}(\mu,x_j)^s,\]
where  the infimum is taken over all finite $(n,r)$-generating sets of $X$, and $g_{r,n}(\mu,x)$ is defined as in the statement of Lemma \ref{asympt1}.
\end{defi}

\begin{propo}
  \label{zerolowerfractal}
Let $\mu\in\M(X)$, $s\in(0,1)$ and $r>0$. Then, 
\[d_{\mu}^{-}(s,r):= \liminf_{n\to\infty} \frac{\log W_{\mu}(s,r,n)}{(1-s)n}\le\liminf_{n\to\infty} \frac{\log S_{\mu}(s,2r,n)}{(1-s)n}.\]
 Moreover, $\underline{H}_{\mu}(f,s)=\displaystyle\lim_{r\to 0}\liminf_{n\to\infty}\dfrac{\log I_{\mu}(s,r,n)}{(1-s)n}\leq \lim_{r\to 0}d_{\mu}^{-}(s,r)$.
\end{propo}
\begin{proof}
Let $n\in\mathbb{N}$. Then, one has 
\[I_{\mu}(s,2r,n)\leq S_{\mu}(s,r,n)\leq W_{\mu}(s,r,n)\leq S_{\mu}(s,2r,n),\]
from which the results follow. The first inequality above comes from (\ref{eq4}).\, The remaining inequalities come from $\mu(B_n(x,r))^s\leq g_{r,n}(\mu,x)^s\leq \mu(B_n(x,2r))^s$, valid for each $x\in X$.
\end{proof}

\begin{propo}
\label{teo2}
Let $s\in (0,1)$, $r>0$, $n\in\mathbb{N}$ and let $E=\{x_l\}_{l=1}^L$ be a finite $(n,r)$-generating set of $X$. Then, the function
\[H_{E}:\M(X)\longrightarrow (0,\infty),~~~~~ H_{E}(\mu)= \sum_{l=1}^Lg_{r,n}(\mu,x_l)^s,\]
is continuous in the weak topology.
\end{propo}
\begin{proof}
Let $(\mu_n)$ be a sequence in $\M(X)$ such that $\mu_n \rightarrow \mu$. Since, for each $l=1,\ldots,L$, the mapping $\M(X)\ni\mu\mapsto g_{r,n}(\mu,x_l)\in[0,\infty)$ is continuous (by Lemma~\ref{asympt1}), it follows that $H_{E}(\mu)=\sum_{l\in L}g_{r,n}(\mu,x_l)^s$ is also continuous, being a finite sum of continuous functions.
\end{proof}

\begin{propo}
\label{Gdeltazero}
Let $s\in (0,1)$. Then, for each $r>0$, $D_{-}^{*}(r)=\{\mu\in \M(X)\mid\ d^-_{\mu}(s,r)=0\}$ is a $G_{\delta}$ subset of $\M(X)$.
\end{propo}
\begin{proof}
  Let $r>0$ and $n\in\mathbb{N}$. Define $h:\M(X)\rightarrow (0,\infty)$ by the law $h(\mu)=W_{\mu}(s,r,n)=\inf_{E}\sum_{x_j\in E}g_{r,n}(\mu,x_j)^s$ (where the infimum is taken over all finite $(n,r)$-generating sets of $X$), and define $p_{r,n}:(0,\infty)\rightarrow \R$ by the law $p_{r,n}(l)=\dfrac{\log(l)}{(1-s)n}$. Note that, for each $k\in\N$, $p_{r,n}^{-1}((-\infty,1/k))=(0,a_k)$, where $a_k=p_{r,n}^{-1}(1/k)$.
  
  It follows from Proposition \ref{teo2} that $h$ is upper-semicontinuous, and thus, for each $k\in\N$,  $(p_{r,n}\circ h)^{-1}((-\infty,1/k))=h^{-1}\left(p_{r,n}^{-1}((-\infty,1/k))\right)=h^{-1}((0,a_k))$ is open in $\M(X)$. Since 
\begin{eqnarray*}
 D_{-}^{*}(r)&=&\left\{\mu\in \M(X)\mid \liminf_{n\to\infty}\frac{\log W_{\mu}(s,r,n)}{(1-s)n}=0\right\}\\
 &=&\bigcap_{k\in\N} \bigcap_{l\in\N} \bigcup_{t>l} \left\{ \mu\in \M(X) \mid \frac{t\log W_{\mu}(s,r,1/t)}{(1-s)}<\frac{1}{k} \right\}\\
 &=&  \bigcap_{k\in\N} \bigcap_{l\in\N} \bigcup_{t>l} ~(p_{r,1/t}\circ h)^{-1}((-\infty,1/k)),
\end{eqnarray*}
 the result follows. 
\end{proof}

\begin{propo}
\label{dense1}
Let $X$ be a compact metric space, let $f:X\rightarrow X$ be a positively expansive map, assume that $\M_p(f)$ is dense in $\M(f)$, and let $s\in (0,1)$. Then, there exists $\ve>0$ (which only depends on $(X,f)$) such that, for each $t\in(0,\ve)$, 
$D_{-}^{*}(t/2)=\{\mu\in \M(f)\mid d^{-}_\mu(s,t/2)=0\}$
is a dense subset of $\mathcal{M}(f)$.
\end{propo}

\begin{proof}
  Since, by hypothesis, $\M_p(f)$ is a dense subset of $\M(f)$, one just have to show that there exists $\ve>0$ such that, for each $t\in(0,\ve)$ and each $\mu\in\M_p(f)$, $d^{-}_\mu(s,t/2)=0$. So, let $\mu$ be a $f$-periodic measure associated with the $f$-periodic point $x\in X$, whose period is $k_x$. Set $\delta:=\min_{0\leq i\neq j\leq k_x-1}\{d(x_i,x_j)\mid x_l:=f^{l}x,\, l=0,\ldots, k_x-1\}$ and set $A:=\{x,fx,\cdots, f^{k_x-1}x\}$. 

Given that $(X,f)$ is positively expansive, there exists $\ve>0$ (which depends only on $(X,f)$) such that, for each $x\in X$, $\Gamma_\ve(x):=\bigcap_{i\ge 0 }f^{-i}(\overline{B}(f^ix,\ve))=\{x\}$.  

{\textit{Claim.}} For each $t\in(0,\ve)$, each $r\in(0,t)$ and each $x\in X$, there exists $n_0\in\mathbb{N}$ such that, for each $n\ge n_0$, $\overline{B}_n(x,t)\subset B(x,r)$, where $\overline{B}_n(x,t):=\{y\in X\mid d_n(y,x)\le t\}$.

Suppose, by absurd, that there exist $t\in(0,\ve)$, $r\in(0,t)$ and $x\in X$ such that for each $n_0\in\N$, there exists $n\ge n_0$ so that $\overline{B}_n(x,t)\not\subset B(x,r)$.  Since, for each $x\in X$, each $r>0$ and each $n\in\N$, $B_{n+1}(x,r)\subset B_n(x,r)$, the last assertion is equivalent to the statement that there exist $t>0$, $r\in(0,t)$, $x\in X$ and $n_0\in\mathbb{N}$ such that, for each $n\ge n_0$, there exists $y_n\in \overline{B}_n(x,t)\setminus B(x,r)$. 

Given that $\{\overline{B}_{n}(x,r)\setminus B(x,r)\}_{n\ge n_0}$ is a decreasing nested sequence of non-empty compact sets, it follows from Cantor's Intersection Theorem that there exists $y\in X$ 
such that for each $n\in\N$, $y\in \overline{B}_n(x,t)\setminus B(x,r)$.


But then, $\Gamma_\ve(x)=\Gamma_t(x)=\bigcap_{i\ge 0}f^{-i}(\overline{B}(f^ix,t))\supset\{x,y\}$, a contradiction with the fact that $(X,f)$ is positively expansive.



\

We split the proof into two cases.

Case 1. $\delta<\ve$. Fix $t\in (\delta,\ve)$. It follows from Claim that for $r=\delta$, there exists $N\in\N$ such that for each $n\ge N$ and each $z\in A$, $B_n(z,t)\subset B(z,\delta)$ (just take $N=\max\{n_0(t,\delta,z)\mid z\in A\}$). Fix an arbitrary $n\ge N$. Since $X$ is a compact metric space and $C=X\setminus\bigcup_{z\in A}B_n(z,t)$ is closed, $C$ is also compact. Now, let $E_1=\{w_l\}_{w_l\in C}$ be a finite $(n,t)$-generating set of $C$, 
and set $\tilde{E}=E_1\cup A$. By construction and Claim, each $z\in A$ belongs to only one Bowen ball associated with $\tilde{E}$ (namely, $B_n(z,t)$), and then, for each $y\in E_1$, $\mu(B_n(y,t))=0$. 
Thus, 
\begin{eqnarray*}
S_{\mu}(s,t,n)=\inf_{E}\sum_{y\in E} \mu(B_n(y,t))^s\leq\sum_{y\in \tilde{E}} \mu(B_n(y,t))^s\le \sum_{z\in A} \mu({B}_n(z,t))^s \le\sum_{z\in A} \mu(B(z,\delta))^s= k_x^{1-s},
\end{eqnarray*}
from which follows that
\[\frac{\log S_{\nu}(s,t,n)}{(1-s)n}
\leq \frac{\log(k_x^{1-s})}{(1-s)n}=\frac{\log k}{n}.\]
Given that the estimate above follows for every $n\ge N$, one has from Proposition~\ref{zerolowerfractal} that $d^-_{\mu}(s,t/2)=0$.

Now, fix $t\in(0,\delta]$ and let $n\in\N$ 
and set $\tilde{E}$ as before. By construction, each $z\in A$ belongs to only one Bowen ball associated with $\tilde{E}$ (namely, $B_n(z,t)$, since $t\le\delta$), and then, for each $y\in E_1$, $\mu(B_n(y,t))=0$. 
Thus, the result follows as before. 

Case 2. $\ve\le\delta$. Just proceed as in the second part of the proof of Case 1.
\end{proof}

\begin{proof4} Since, by Proposition~\ref{zerolowerfractal},
  \[\mathcal{H}_{-}\supset\bigcap_{m>[2/\ve]}D^*_{-}(1/m),\]
  the result follows from Propositions~\ref{Gdeltazero} and~\ref{dense1}.
\end{proof4}

\subsection{Lipshitz maps}

Suppose now that $f$ is a Lipshitz map with constant $\Lambda>1$. Let $\mu\in\M(X)$, let $s\in(0,1)$ and let $\mathcal{G}=\{B(x_j,t)\}$ be some countable covering of $X$ by balls of radius $t>0$. Let $\tilde{\mathcal{G}}=\{B(x_i,t)\}\subset \mathcal{G}$ be a sub-covering of $X$ that also covers $\supp(\mu)$.

It is easy to show that for each $x\in X$, each $r>0$ and each $n\in\N$, $B(x,r\Lambda^{-n})\subset B_n(x,r)$. Note also that, for each  $x\in B(y,t)$, one has $B(y,t)\subset B(x,2t)$, from which follows that, for each $x\in B(x_i,r\Lambda^{-n})\cap\supp(\mu)$, $\mu(B(x_i,r\Lambda^{-n}))^{s-1}\ge \mu(B(x,2r\Lambda^{-n}))^{s-1}$; hence,
\begin{eqnarray}\label{ineqimp}
I_{\mu}(s,2r,n)\nonumber
&= & \int_{\supp(\mu)}\mu(B_n(x,2r))^{s-1}d\mu(x)\nonumber 
\leq  \sum_{x_i\in \tilde{\mathcal{G}}}\int_{B(x_i,r\Lambda^{-n})\cap\supp(\mu)} \mu(B(x,2r\Lambda^{-n}))^{s-1}d\mu(x)\nonumber \\
&\leq & \sum_{x_i\in \tilde{\mathcal{G}}}\int_{B(x_i,r\Lambda^{-n})\cap\supp(\mu)} \mu(B(x_i,r\Lambda^{-n}))^{s-1}d\mu(x)
=\nonumber \sum_{x_i\in \tilde{\mathcal{G}}} \mu(B(x_i,r\Lambda^{-n}))^s \\
&\leq &\sum_{x_j\in \mathcal{G}} \mu(B(x_j,r\Lambda^{-n}))^s
\end{eqnarray}
(by $x\in \mathcal{G}$, one means that $B(x,r\Lambda^{-n})\in\mathcal{G}$). 

Now, we may combine Propositions~2.3 and~3.1 in~\cite{AS1} in order to prove the following result.

\begin{propo}\label{zerogeneric} Let $s\in (0,1)$ and assume that $\M_p(f)$ is dense in $\M(f)$. Then, for each $r\in(0,1)$, $D_{-}^{*}(r)=\{\mu\in \M(f)\mid\ d^-_{\mu}(s,r)=0\}$ is a dense $G_{\delta}$ subset of $\M(f)$.
\end{propo}
\begin{proof} The fact that, for each $r\in(0,1)$, $D_{-}^{*}(r)$ is a $G_\delta$ subset of $\M(X)$ is just (a modification of) Proposition~2.3 in~\cite{AS1}. 

  Now, using the same notation as in the proof of Proposition~\ref{dense1}, fix $n>\max\{0,(\log r-\log\delta)/\log\Lambda\}$, let $\mathcal{G}_1=\{B(y_m,r\Lambda^{-n})\}_{y_m\in C}$ be a finite covering of $C=X\setminus\bigcup_{z\in A}B(z,r\Lambda^{-n})$,  and set $\mathcal{\tilde{G}}:=\mathcal{G}_1\cup\{B(z,r\Lambda^{-n})\}_{z\in A}$. By construction, each $z\in A$ belongs to only one element of $\mathcal{\tilde{G}}$ (namely, $B(z,r\Lambda^{-n})$), and for each $y_m\in\mathcal{G}_1$, $\mu(B(y_m,r\Lambda^{-n}))=0$. 

Thus, 
\begin{eqnarray*}
S_{\mu}(s,r,n)=\inf_{\mathcal{G}}\sum_{z_j\in \mathcal{G}} \mu(B(z_j,r\Lambda^{-n}))^s\leq \sum_{w\in \mathcal{\tilde{G}}} \mu(B(w,r\Lambda^{-n}))^s =  k_x^{1-s},
\end{eqnarray*}
from which follows, by letting $n\to\infty$, that $d^-_{\mu}(s,r)=0$.
\end{proof}

\begin{proof5} Since, by relation~\eqref{ineqimp} and Proposition~\ref{zerolowerfractal},
  \[\mathcal{H}_{-}\supset\bigcap_{m>[2/\ve]}D^*_{-}(1/m),\]
 the result follows from Proposition~\ref{zerogeneric}. 
\end{proof5}

\section{Expansive measures}
\label{expasivemeasures}

Before we present the proof of Theorem~\ref{GExpansive}, some preparation is required.

Given a bijective map $f: X \rightarrow X$, $x \in X$, $\delta>0$ and $n \in \mathbb{N}^{+}$, one defines the so-called two-sized (closed) Bowen ball of size $n$ and radius $\delta$, centered at $x$ as
\[V[x, n, \delta]=\left\{y \in X\mid d\left(f^{i}(x), f^{i}(y)\right) \leq \delta,\; -n \leq i \leq n\right\},
\]
that is, 
\[V[x, n, \delta]=\bigcap_{i=-n}^{n} f^{-i}\left(\overline{B}\left(f^{i}(x), \delta\right)\right).\]

The next result characterizes an expansive measure in terms of two-sized Bowen balls.

\begin{lema}[Lemma 1.16 in \cite{Moraleslibro}]
\label{equivalexpansive}
Let $f: X \rightarrow X$ be a homeomorphism of a metric
space $X$.  A Borel probability measure $\mu$ on $X$ is an expansive measure of
$f$ if and only if there exists $\delta>0$ such that, for $\mu$-a.e. $x\in X$,
\begin{eqnarray}
\liminf _{n \rightarrow \infty} \mu(V[x, n, \delta])=0. 
\end{eqnarray}
\end{lema}



The proof of Theorem~\ref{GExpansive} follows the same steps of the proof of Theorem~\ref{central1}. Namely, we set, 
for each $x\in X$ and each $\ve>0$,
\begin{equation}\label{beta}
  \underline{\omega}_{\mu}^\ve(f,x):=\liminf_{n\to\infty}\mu(V[x, n, \ve])=\lim_{s\to \infty}\inf_{n\geq s} \mu(V[x, n, \ve])=\lim_{s\to \infty}\underline{\eta}_{\mu}^\ve(x,s);
\end{equation}
note that, for each $\ve>0$, $\N\ni s\mapsto\underline{\eta}_{\mu}^\ve(x,s)\in[0,\infty]$ is a non-decreasing function, whereas, for each $s\in\N$, $\mathbb{R}_+\ni \ve\mapsto\underline{\eta}_{\mu}^\ve(x,s)\in[0,\infty]$ is a non-increasing function.

Let, for each $x\in X$, each $\ve>0$ and each $s\in\N\cup\{0\}$,
\[\underline{\theta}_\mu^\ve(x,s):=\inf_{n\ge s}g_{x,\ve,n}(\mu),\]
where $g_{x,\ve,n}(\mu)$ is defined as in Lemma~\ref{asympt1}, replacing $d_n(\cdot,\cdot)$ by $D_n:X\times X\rightarrow\mathbb{R}_+$, $D_n(x,y)=\max\{d(f^kx,f^ky)\mid-n\le k\le n\}$.

\begin{lema}
\label{Gdeltaexpansive}
Let $f: X \rightarrow X$ be a uniform homeomorphism (that is, $f$ and $f^{-1}$ are uniformly continuous functions) of a Polish metric
space $X$, let $\alpha,\ve>0$ and $s\in\N\cup\{0\}$. Then, the set 
\begin{eqnarray}
\mathcal{EM}(\alpha,\ve,s):=\{ \mu\in \M(X)\mid\mu\textrm{-}\esssup\underline{\theta}_{\mu}^\ve(x,s)\le\alpha\}
\end{eqnarray}
 is a $G_{\delta}$ subset of $\M(X)$.
\end{lema}
\begin{proof}
%
The proof follows from the same arguments presented in the proof of Lemma~\ref{Gdelta3}, replacing $\underline{\gamma}_{\mu}^\ve(x,s)$ by $\underline{\theta}_{\mu}^\ve(x,s)$.
\end{proof}

\begin{proof3}
  One has, by Lemma~\ref{equivalexpansive} and relation~\eqref{beta},  the following characterization of $\M_{exp}$:
\begin{eqnarray*}
\M_{exp}&=&\{\mu\in \M(X)\mid \exists \,\ve >0 \mbox{ such that } \liminf _{n \rightarrow \infty} \mu(V[x, n, \ve])=0, ~ \text{ for } \mu\text{-}a.e.~ x \in X\}\\
&=& \bigcup_{k=1}^{\infty} \{\mu\in \M(X)\mid  \liminf _{n \rightarrow \infty} \mu(V[x, n, 1/k])=0, ~ \text{ for } \mu\text{-}a.e.~ x \in X\}\\
&=& \bigcup_{k=1}^{\infty} \bigcap_{l\in \N} \{ \mu\in \M(X)\mid\liminf _{n \rightarrow \infty} \mu(V[x, n, 1/k])\le 1/l, ~ \text{ for } \mu\text{-}a.e.~ x \in X\}\\
&=& \bigcup_{k=1}^{\infty} \bigcap_{l\in \N}  \{ \mu\in \M(X)\mid\mu\textrm{-}\esssup\underline{\omega}_{\mu}^{1/k}(f,x)\le 1/l\}\\
&=& \bigcup_{k=1}^{\infty} \bigcap_{l\in \N} \bigcap_{s\in \N} \{ \mu\in \M(X)\mid\mu\textrm{-}\esssup\underline{\eta}_{\mu}^{1/k}(x,s)\le 1/l\}\\
&=& \bigcup_{k=1}^{\infty} \bigcap_{l\in \N} \bigcap_{s\in \N} \{ \mu\in \M(X)\mid\mu\textrm{-}\esssup\underline{\theta}_{\mu}^{1/k}(x,s)\le 1/l\}.
\end{eqnarray*}
The result follows now from Lemma~\ref{Gdeltaexpansive}.
\end{proof3}

\bibliography{refsentropy}{}
\bibliographystyle{acm}

\end{document}